\documentclass{article}
\usepackage{lmodern}

\usepackage[T1]{fontenc}
\usepackage[latin9]{inputenc}
\usepackage{geometry}
\usepackage[english]{babel}
\usepackage{array}
\usepackage{verbatim}
\usepackage{color}

\usepackage{float}
\usepackage{varwidth}
\usepackage{amsmath}
\usepackage{amsthm}
\usepackage{amssymb}
\usepackage[round]{natbib}
\usepackage[pdfusetitle,
 bookmarks=true,bookmarksnumbered=false,bookmarksopen=false,
 breaklinks=false,pdfborder={0 0 0},pdfborderstyle={},backref=false,colorlinks=false]
 {hyperref}

\makeatletter

\providecommand{\tabularnewline}{\\}
\newenvironment{cellvarwidth}[1][t]
    {\begin{varwidth}[#1]{\linewidth}}
    {\@finalstrut\@arstrutbox\end{varwidth}}

\theoremstyle{plain}
\newtheorem{thm}{\protect\theoremname}[section]
\theoremstyle{definition}
\newtheorem{example}[thm]{\protect\examplename}
\theoremstyle{remark}
\newtheorem{rem}[thm]{\protect\remarkname}

\newtheorem{assumption}{Assumption}

\makeatother

\providecommand{\examplename}{Example}
\providecommand{\remarkname}{Remark}
\providecommand{\theoremname}{Theorem}

\begin{document}
\title{Precise quantile function estimation from the characteristic function}
\author{Gero Junike\thanks{Corresponding author. Carl von Ossietzky Universität, Institut für
Mathematik, 26129 Oldenburg, Germany, ORCID: 0000-0001-8686-2661,
Phone: +49 441 798-3729, E-mail: gero.junike@uol.de\\ \textcolor{blue}{This article is accepted in the journal Statistics
\& Probability Letters.}}}
\maketitle
\begin{abstract}
We provide theoretical error bounds for the accurate numerical computation
of the quantile function given the characteristic function of a continuous
random variable. We show theoretically and empirically that the numerical
error of the quantile function is typically several orders of magnitude
larger than the numerical error of the cumulative distribution function
for probabilities close to zero or one. We introduce the COS method
for computing the quantile function. This method converges exponentially
when the density is smooth and has semi-heavy tails and all parameters
necessary to tune the COS method are given explicitly. Finally, we
numerically test our theoretical results on the normal-inverse Gaussian
and the tempered stable distributions.
\end{abstract}

\section{\protect\label{sec:Introduction}Introduction}

Let $(\Omega,\mathcal{F},P)$ be a probability space and $X:\Omega\to\mathbb{R}$
be a random variable with density $f$, cumulative distribution function
(CDF) $F$, quantile function (QF) $F^{-1}$ and characteristic function
(CF) $\varphi$. In this research, we are interested in the precise
numerical computation of the QF, provided that the CF is given in
closed form. QFs are used, for example, in Monte Carlo simulations
to generate random numbers using the fact that $X$ and $F^{-1}(U)$
are equal in distribution where $U$ is uniformly distributed. For
the tempered stable (TS) distribution, for example, ``neither the
density function nor specific random number generators are available'',
see \citet[Sec. 8.4.3, p. 111]{schoutens2003levy}. Therefore, a precise
approximation of the QF of the TS distribution is useful for random
number generation. Applications of the QF in statistics are discussed
in \citet{gilchrist2000statistical}.

Typically, one first calculates $F$ using the Gil-Pelaez formula
and then inverts $F$ numerically. The Gil-Pelaez formula is stated,
for example, in \citet[Corollary 2]{hughett1998error} and \citet[Eq. (3.6)]{abate1992fourier}
and reads
\begin{equation}
F(y)=\frac{1}{2}-\int_{-\infty}^{\infty}\frac{\varphi(u)}{2i\pi u}e^{-iuy}du,\quad y\in\mathbb{R},\quad\text{and}\quad F(y)=\frac{2}{\pi}\int_{0}^{\infty}\Re\{\varphi(u)\}\frac{\sin(yu)}{u}du,\quad y\geq0\label{eq:Gil}
\end{equation}
for CDFs with full support and support on the positive reals, respectively.
The integrals in (\ref{eq:Gil}) must be solved numerically. In this
research we use the COS method, which is introduced in Section \ref{sec:Theory},
to approximate $F$ from $\varphi$ since all parameters necessary
to tune the COS method can be obtained directly from $\varphi$, and
the COS method converges exponentially, provided $f$ is smooth and
has semi-heavy tails. Furthermore, the COS method makes it possible
to approximate $F$ and $f$ simultaneously, which will be helpful
in estimating the error on the QF. When $f$ has heavy tails, e.g.,
when $f$ belongs to the stable law, other Fourier inversion methods
-- such as the Gil-Pelaez formula or the Carr-Madan formula (see
\citet{carr1999option}) -- are numerically more efficient, see \citet{junike2023handle}.
A robust version of the COS method for unbounded functions can be
found in \citet{wang2017pricing}. A direct link between $\varphi$
and $F^{-1}$ via non-linear integro-differential equations is given
in \citet{shaw2009monte}. Suppose $H$ is a numerical approximation
of $F$, in the sense that
\begin{equation}
\sup_{y\in\mathbb{R}}|F(y)-H(y)|\leq\varepsilon,\label{eq:F-H}
\end{equation}
given some predefined error tolerance $\varepsilon>0$. Depending
on the exact Fourier technique and the numerical integration method,
$H$ depends on parameters such as truncation range, number of terms,
step size and so on. In the case of the Gil-Pelaez formula and the
COS method, bounds on these parameters are given explicitly, see \citet{abate1992fourier}
and \citet{junike2023handle}. That is, given $\varepsilon$, it is
well known how to construct $H$ such that Inequality (\ref{eq:F-H})
holds. Let $p\in(0,1)$ and $\delta>0$. To the best of our knowledge,
however, it is not known how to choose $\varepsilon$ such that Inequality
(\ref{eq:F-H}) implies 
\begin{equation}
|F^{-1}(p)-H_{\text{Num}}^{-1}(p)|<\delta,\label{eq:F^-1}
\end{equation}
where $H^{-1}$ is the (exact) inverse of $H$ and $H_{\text{Num}}^{-1}$
is the approximation of $H^{-1}$ by a numerical inversion technique.
In our main Theorem \ref{thm:HH-1} we show that the error between
$F^{-1}$ and $H_{\text{Num}}^{-1}$ depends linearly on the error
between $H^{-1}$ and $H_{\text{Num}}^{-1}$, linearly on the error
between $F$ and $H$ and reciprocally on the derivative $h:=H^{\prime}$.
In particular, in the tails, $h$ tends to zero, i.e., we show theoretically
that for $p$ close to zero or close to one, $\varepsilon$ must be
several orders of magnitude smaller than $\delta$ to ensure that
Inequality (\ref{eq:F-H}) implies Inequality (\ref{eq:F^-1}). 

We confirm the theoretical results by numerical experiments on distributions
with support on $(-\infty,\infty)$ and $(0,\infty)$. In particular,
we consider the normal distribution, the normal-inverse Gaussian (NIG)
distribution (see \citet{barndorff1978hyperbolic,barndorff1997normal})
and the TS distribution (see \citet{hougaard1986survival}). 

The CDF and the QF of the NIG and the TS distributions are not known
in closed form. The density of the NIG distribution can be expressed
in terms of the modified Bessel function of the third kind, but for
the density of the TS distribution only an infinite series representation
is known. 

This letter is organized as follows: in Section \ref{sec:Theory}
we discuss our main results: how the numerical error on the CDF propagates
to the QF and we introduce the COS method. In Section \ref{sec:Numerical-experiments}
we present numerical experiments confirming the theoretical results.

\section{\protect\label{sec:Theory}The COS method and QF by CF}

Let $F$ be a differentiable CDF and $H:\mathbb{R}\to\mathbb{R}$
be a differentiable function, which serves as an approximation of
$F$. Let $f=F^{\prime}$ and $h=H^{\prime}$. We make the following
assumptions:\begin{assumption}\label{A1} 	
There is an interval $(\alpha,\beta)\subset\overline{\mathbb{R}}$ such that $f(x)>0$
if and only if $x\in(\alpha,\beta)$.
\end{assumption}\begin{assumption}\label{A2} 	
There is an interval $(a,b)\subset\overline{\mathbb{R}}$ such that $h(x)>0$ if and only if $x\in(a,b)$.
\end{assumption}These intervals can be chosen to be open, since densities are only
almost surely unique. We then have that $F$ and $H$ are bijective
maps from $(\alpha,\beta)$ to $(0,1)$ and $(a,b)$ to $(0,1)$,
respectively. The inverse function of $F$ and $H$ are denoted by
$F^{-1}$ and $H^{-1}$. We introduce another error by numerically
inverting $H$. We denote the approximation of $H^{-1}$ by $H_{\text{Num}}^{-1}$.
The next theorem explains how the bias between $F$ and $H$ and $H^{-1}$
and $H_{\text{Num}}^{-1}$ propagates when $F^{-1}$ is approximated
by $H_{\text{Num}}^{-1}$.
\begin{thm}
\label{thm:QF}Assume Assumptions \ref{A1} and \ref{A2} hold. Let
$p\in(0,1)$ and $\varepsilon>0$ with $0<p\pm\varepsilon<1$. Assume
$\sup_{y\in\mathbb{R}}|F(y)-H(y)|\leq\varepsilon$ and $|H^{-1}(p)-H_{\text{Num}}^{-1}(p)|\leq\varepsilon$.
Let $y=H_{\text{Num}}^{-1}(p)$. Then it holds for some $c\in[-\varepsilon,\varepsilon]$
that $h(y+c)>0$ and
\begin{equation}
|F^{-1}(p)-H_{\text{Num}}^{-1}(p)|\leq\frac{2\varepsilon}{h(y+c)}+\varepsilon+o(\varepsilon).\label{eq:H-1}
\end{equation}
\end{thm}

\begin{proof}
We use
\begin{equation}
|F^{-1}(p)-H_{\text{Num}}^{-1}(p)|\leq|F^{-1}(p)-H^{-1}(p)|+|H^{-1}(p)-H_{\text{Num}}^{-1}(p)|.\label{eq:HNum}
\end{equation}
The second term at the right-hand side of Inequality (\ref{eq:HNum})
is less or equal than $\varepsilon$ by assumption, which also implies
that there is a $c\in[-\varepsilon,\varepsilon]$ with $H^{-1}(p)=y+c$.
By Assumption \ref{A1}, it holds that $h(y+c)>0$. We analyze the
first term at the right-hand side of Inequality (\ref{eq:HNum}):
We have for $\tilde{y}:=H^{-1}(p+\varepsilon)$ that $F(\tilde{y})\geq H(\tilde{y})-\varepsilon=p$.
So, $H^{-1}(p+\varepsilon)\geq F^{-1}(p)$. Similarly, $F^{-1}(p)\geq H^{-1}(p-\varepsilon)$
holds. By the monotonicity of $H^{-1}$ it holds that $H^{-1}(p+\varepsilon)\geq H^{-1}(p)\geq H^{-1}(p-\varepsilon)$.
Therefore,
\[
|H^{-1}(p)-F^{-1}(p)|\leq H^{-1}(p+\varepsilon)-H^{-1}(p-\varepsilon).
\]
Next, we use the inverse function rule to conclude that
\[
H^{-1}(p\pm\varepsilon)=H^{-1}(p)\pm\frac{\varepsilon}{h\big(H^{-1}(p)\big)}+o(\varepsilon).
\]
Hence, $H^{-1}(p+\varepsilon)-H^{-1}(p-\varepsilon)=\frac{2\varepsilon}{h(y+c)}+o(\varepsilon)$,
which completes the proof.
\end{proof}
In order to apply Theorem \ref{thm:QF}, we have to compute $H$ such
that the absolute difference between $H$ and $F$ is small. In this
research we use the COS method to obtain $H$ from the CF $\varphi$,
however, there are other Fourier inversion techniques to obtain $H$
from $\varphi$, e.g., the Gil-Pelaez formula.

To apply the COS method we assume that $f$ has semi-heavy or lighter
tails, which implies that the COS method converges exponentially,
see \citet{junike2023handle}. Formally, we make the following assumption: 

\begin{assumption}\label{A3} 	
For constants $C_{1},C_{2},L_{0}>0$ we assume that $|f(x)|\leq C_{1}\exp(-C_{2}|x|),\quad|x|\geq L_{0}.$
\end{assumption}
\begin{example}
The Generalized Hyperbolic distribution has semi-heavy tails and support
on $(-\infty,\infty)$. It contains many other distributions as special
cases, e.g. the NIG, the Variance Gamma and the Hyperbolic distribution,
see \citet[Sec. 5.3.11]{schoutens2003levy}. The TS distribution has
semi-heavy tails and support on $(0,\infty)$. It includes the normal-inverse
Gaussian%
\begin{comment}
distribution discussed in \citet{folks1978inverse}
\end{comment}
{} and the Gamma distribution as special cases. Densities with heavy
tails that do not meet Assumption \ref{A3} are, for example, the
Pareto and stable distributions. 
\end{example}

Next, we introduce the COS method, see \citet{fang2009novel}. Let
$(\alpha,\beta)\subset\overline{\mathbb{R}}$ such that Assumption
\ref{A1} is satisfied. Let $(a,b)\subset\mathbb{R}$ be a large but
finite interval with $(a,b)\subset(\alpha,\beta)$. Let $N\in\mathbb{N}$
be large enough. Since only $\varphi$ is given, we approximate $f$
as follows: first we truncate $f$, then the truncated density is
approximated by a classical Fourier-Cosine series, i.e.,
\[
f(x)\approx f(x)1_{(a,b)}(x)\approx\frac{a_{0}}{2}+\sum_{k=1}^{N}a_{k}\cos\left(k\pi\frac{x-a}{b-a}\right)\approx\frac{c_{0}}{2}+\sum_{k=1}^{N}c_{k}\cos\left(k\pi\frac{x-a}{b-a}\right)=:h_{\text{COS}}(x).
\]
The coefficients $a_{k}$ are defined and approximated as follows:
\begin{align*}
a_{k} & :=\frac{2}{b-a}\int_{a}^{b}f(x)\cos\left(k\pi\frac{x-a}{b-a}\right)dx\\
 & \approx\frac{2}{b-a}\int_{\alpha}^{\beta}f(x)\cos\left(k\pi\frac{x-a}{b-a}\right)dx\\
 & =\frac{2}{b-a}\Re\left\{ \varphi\left(\frac{k\pi}{b-a}\right)\exp\left(-i\frac{ka\pi}{b-a}\right)\right\} =:c_{k}.
\end{align*}
For two real numbers $x$ and $y$ we denote by $x\land y:=\min(x,y)$
and by $x\lor y:=\max(x,y)$. Given $h_{\text{COS}}$, we approximate
$F(y)=\int_{-\infty}^{y}f(x)dx$ by zero for $y\leq a$ and for $y>a$
by
\begin{align*}
F(y) & \approx\int_{a}^{y\land b}h_{\text{COS}}(x)dx=\frac{c_{0}}{2}(y\land b-a)+\sum_{k=1}^{N}c_{k}\frac{b-a}{k\pi}\sin\left(k\pi\frac{y\land b-a}{b-a}\right)=:H_{\text{COS}}(y).
\end{align*}

Since $f(x)>0$, observe that $h_{\text{COS}}(x)>0$ for $x\in(a,b)$
and $N$ large enough if $f$ is continuous and piecewise continuously
differentiable on $(\alpha,\beta)$ since then the Fourier-Cosine
series converges pointwise. Further, $\int_{-\infty}^{\infty}h_{\text{COS}}(x)dx=\frac{c_{0}}{2}(b-a)=1$.
So, $h_{\text{COS}}$ is a density and $H_{\text{COS}}$ is a CDF,
which is bijective as a map from $(a,b)$ to $(0,1)$. In particular,
$h_{\text{COS}}$ satisfies Assumption \ref{A2}. The next Theorem
gives conditions on $(a,b)$ such that $\sup_{y\in\mathbb{R}}|H_{\text{COS}}(y)-F(y)|\leq\varepsilon$,
which is an essential assumption in Theorem \ref{thm:QF}.
\begin{thm}
\label{thm:HH-1}Assume $f$ is a bounded density satisfying Assumption
\ref{A3}. Let $\varepsilon>0$ be small enough. Let $N\in\mathbb{N}$
be large enough. Let $n\in\mathbb{N}$ be even and set $\mu:=E[X]$
and
\[
\ell:=\sqrt[n]{\frac{2E[(X-\mu)^{n}]}{\varepsilon}},\quad a:=(\mu-\ell)\lor\alpha,\quad b:=(\mu+\ell)\land\beta.
\]
It then follows that $\sup_{y\in\mathbb{R}}|H_{\text{COS}}(y)-F(y)|\leq\varepsilon$. 
\end{thm}

\begin{proof}
The inequality $|H_{\text{COS}}(y)-F(y)|<\varepsilon$ for all $y\in\mathbb{R}$
follows as in \citet[Corollary 9]{junike2022precise} using Markov's
inequality and the fact that $f$ has semi-heavy tails. 
\end{proof}
In the following Remarks, we provide more details on how to implement
Theorems \ref{thm:QF} and \ref{thm:HH-1}.
\begin{rem}
In practical applications, we suggest replacing the right-hand side
in Inequality (\ref{eq:H-1}) by $\frac{2\varepsilon}{\min\{h(y\pm\varepsilon)\}}+\varepsilon$.
Observe that $E[(X-\mu)^{n}]=\frac{1}{i^{n}}\frac{\partial}{\partial u^{n}}\left.\psi(u)\right|_{u=0}$,
where $\psi$ is the CF of $X-\mu$, i.e., $\psi(u)=\varphi(u)e^{-iu\mu}$.
So, we need only obtain the $n^{th}$-derivative of $\psi$ to compute
\emph{$E[(X-\mu)^{n}]$}. \citet{junike2022precise} suggest choosing
$n\in\{4,6,8\}$. In our experiments, we set $n=8$.
\end{rem}

\begin{rem}
We suggest a root-finding algorithm, e.g., Newton's method, the secant
method, the fixed point iteration method or the bisection method,
to invert $H$. The bisection method has the advantage of providing
a full error control, i.e., we are able to compute $H_{\text{Num}}^{-1}$
such that $|H^{-1}(p)-H_{\text{Num}}^{-1}(p)|\leq\varepsilon$ holds.
The method repeatedly bisects the interval $(a,b)$ by selecting the
subinterval in which the function $H(\cdot)-p$ changes its sign until
the bisected interval has a length less than $\varepsilon$.
\end{rem}

\begin{rem}
Let $\delta>0$ and $p\in(0,1)$ be given. Suppose we need $|H_{\text{Num}}^{-1}(p)-F^{-1}(p)|\leq\delta$.
How should we choose the error tolerance $\varepsilon$ for the CDF?
This is a tricky question: the choice of $\varepsilon$ affects the
truncation range $(a,b)$ and thus $h$ and $H_{\text{Num}}^{-1}$.
However, the right-hand side of Inequality (\ref{eq:H-1}) also depends
on $h$. We suggest starting with a reasonable error tolerance $\varepsilon$.
Then, compute $h$ and $H_{\text{Num}}^{-1}$ and check if the inequality
$\frac{2\varepsilon}{h(y)}+2\varepsilon\leq\delta$ is satisfied.
If it is not satisfied, reduce $\varepsilon$ successively until the
inequality $\frac{2\varepsilon}{h(y)}+2\varepsilon\leq\delta$ holds.
Then we can be sure that $|H_{\text{Num}}^{-1}(p)-F^{-1}(p)|\leq\delta$.
\end{rem}

\begin{rem}
\label{rem:N}Let $L:=\frac{b-a}{2}$ and $s\in\mathbb{N}$ be odd.
If $f$ is $s+1$ times differentiable with bounded derivatives and
the derivatives also have semi-heavy tails, then the number of terms
in Theorem \ref{thm:HH-1} can be determined by 
\[
N\geq\left(\frac{1}{\pi}\int_{0}^{\infty}u{}^{s+1}|\varphi(u)|du\right)^{\frac{1}{s}}\times\left(\frac{2^{s+\frac{5}{2}}L^{s+2}}{s\pi^{s+1}}\frac{12}{\varepsilon}\right)^{\frac{1}{s}},
\]
see \citet[Eq. (3.8)]{junike2023handle}. The last integral can be
solved numerically using standard techniques, e.g., Gauss--Laguerre
quadrature, and in some cases (e.g., normal and NIG distributions)
it is given explicitly. \citet{junike2023handle} suggests $s\in\{19,...,39\}$.
In our experiments, we set $s=39$. There is also an implicit way
to find $N$ without additional smoothness assumptions on $f$, but
it requires that $\int_{-\infty}^{\infty}|\varphi(u)|^{2}du$ is given
exactly, see \citet[Corollary 3.12]{junike2024multidimensional}. 
\end{rem}

\section{\protect\label{sec:Numerical-experiments}Numerical experiments}

In our numerical experiments, we consider three random variables,
$X_{1}$, $X_{2}$ and $X_{3}$, and compute their quantiles via Theorems
\ref{thm:QF} and \ref{thm:HH-1}. By $X_{1}$ we denote the standard
normal random variable, abbreviated as N(0,1), which has mean $0$,
variance $1$, skewness $0$ and kurtosis $3$.

$X_{2}$ is a TS distributed random variable with parameters $c>0$,
$d\geq0$ and $\kappa\in(0,1)$, which has the characteristic function
$u\mapsto\exp\big(cd-c\big(d^{\frac{1}{\kappa}}-2iu\big)^{\kappa}\big)$.
We set $c=d=1$ and $\kappa=\frac{3}{4}$. The random variable $X_{2}$
has support on $(0,\infty)$, mean $1.5$, variance $0.75$, skewness
$2.89$ and kurtosis $18$. 

$X_{3}$ follows a NIG distribution with parameters $\gamma>0$, $\theta\in(-\gamma,\gamma)$
and $\nu>0$, which is defined as a normal variance-mean mixture where
the mixing density is the inverse Gaussian distribution. The random
variable $X_{3}$ has characteristic function $u\mapsto\exp\big(-\nu\big(\sqrt{\gamma^{2}-(\theta+iu){}^{2}}-\sqrt{\gamma^{2}-\theta^{2}}\big)\big)$.
We set $\nu=\gamma=1$ and $\theta=0$. Then $X_{3}$ has support
on $\mathbb{R}$ and has mean $0$, variance $1$, skewness $0$ and
kurtosis $6$, i.e., much heavier tails than the normal distribution.
\begin{rem}
On the computation of reference values: We compute reference values
for Table \ref{tab:numerics} for $F$ by the COS method using $(a,b)$
as in Theorem \ref{thm:HH-1} with $\varepsilon=10^{-9}$. We set
$N=10^{7}$. We confirm the reference values by the the Gil-Palaez
formula, see (\ref{eq:Gil}). The values for the CDF using the COS
method and the Gil-Palaez formula agree up to $12$ digits. We then
apply the bisection method for numerical inversion with error tolerance
$\varepsilon=10^{-9}$ to obtain a reference value for $F^{-1}$.
Theorem \ref{thm:QF} ensures that the reference value for $F^{-1}$
and the true QF coincide up to 6 digits. In the case of the normal
distribution, we double check that the reference values for $F$ and
$F^{-1}$ agree with the known closed form solutions up to 9 digits.
\end{rem}

Table \ref{tab:numerics} shows the parameters $a$, $b$ and $N$
of the COS method for different error tolerances of $\varepsilon$
for the three distributions. For different probabilities $p\in(0,1)$,
we compute $y:=H_{\text{Num}}^{-1}(p)$, $F(y)$, $|H(y)-F(y)|$,
$|H_{\text{Num}}^{-1}(p)-F^{-1}(p)|$ and $h(y)$. We observe that
the right-hand side (RHS) of Inequality (\ref{eq:H-1}) is always
satisfied and to some extent overestimates the true error on the QF,
since the inequalities in the proofs of Theorems \ref{thm:QF} and
\ref{thm:HH-1} are not sharp for the three distributions. 

\begin{table}[H]
\begin{centering}
\begin{tabular}{|>{\centering}p{0.8cm}>{\centering}p{0.7cm}>{\centering}p{0.7cm}>{\centering}p{0.4cm}>{\centering}p{0.4cm}>{\centering}p{1.2cm}>{\centering}p{1cm}>{\centering}p{1.25cm}>{\centering}p{1.3cm}>{\centering}p{1.5cm}c|}
\hline 
\multicolumn{1}{|>{\centering}p{0.8cm}|}{$F$} & \multicolumn{1}{>{\centering}p{0.7cm}|}{$\varepsilon$} & \multicolumn{1}{>{\centering}p{0.7cm}|}{$b-a$} & \multicolumn{1}{>{\centering}p{0.4cm}|}{$N$} & \multicolumn{1}{>{\centering}p{0.4cm}|}{$p$} & \multicolumn{1}{>{\centering}p{1.2cm}|}{$y:=$

$H_{\text{Num}}^{-1}(p)$} & \multicolumn{1}{>{\centering}p{1cm}|}{$F(y)$} & \multicolumn{1}{>{\centering}p{1.25cm}|}{$|H(y)$

$-F(y)|$} & \multicolumn{1}{>{\centering}p{1.3cm}|}{$|H_{\text{Num}}^{-1}(p)$

$-F^{-1}(p)|$} & \multicolumn{1}{>{\centering}p{1.5cm}|}{$h(y-\varepsilon)\land h(y+\varepsilon)$} & \begin{cellvarwidth}[t]
\centering
RHS

of (\ref{eq:H-1})
\end{cellvarwidth}\tabularnewline
\hline 
\hline 
TS & 0.005 & 10.2 & 50 & 0.01 & 0.01492 & 0.01475 & 0.00013 & 0.00481 & 0.980 & 0.02\tabularnewline
TS & 0.005 & 10.2 & 50 & 0.1 & 0.10444 & 0.10370 & 0.00055 & 0.00370 & 1.000 & 0.02\tabularnewline
TS & 0.005 & 10.2 & 50 & 0.25 & 0.24370 & 0.24619 & 0.00039 & 0.00362 & 1.056 & 0.01\tabularnewline
TS & 0.005 & 10.2 & 50 & 0.75 & 0.70127 & 0.75166 & 0.00077 & 0.00172 & 0.940 & 0.02\tabularnewline
TS & 0.005 & 10.2 & 50 & 0.9 & 0.98974 & 0.89978 & 0.00084 & 0.00156 & 0.142 & 0.08\tabularnewline
TS & 0.005 & 10.2 & 50 & 0.99 & 3.36711 & 0.98996 & 0.00007 & 0.00504 & 0.009 & 1.14\tabularnewline
\hline 
N(0,1) & 0.005 & 7.6 & 12 & 0.75 & 0.67617 & 0.75053 & 0.00000 & 0.00168 & 0.316 & 0.04\tabularnewline
N(0,1) & 0.005 & 7.6 & 12 & 0.9 & 1.28214 & 0.90010 & 0.00000 & 0.00059 & 0.174 & 0.06\tabularnewline
N(0,1) & 0.005 & 7.6 & 12 & 0.99 & 2.32411 & 0.98994 & 0.00000 & 0.00224 & 0.026 & 0.38\tabularnewline
\hline 
NIG & 0.005 & 11.9 & 79 & 0.75 & 0.53675 & 0.74896 & 0.00000 & 0.00284 & 0.365 & 0.03\tabularnewline
NIG & 0.005 & 11.9 & 79 & 0.9 & 1.14023 & 0.90019 & 0.00000 & 0.00124 & 0.153 & 0.07\tabularnewline
NIG & 0.005 & 11.9 & 79 & 0.99 & 2.70116 & 0.98999 & 0.00000 & 0.00073 & 0.014 & 0.73\tabularnewline
NIG & 0.0005 & 15.8 & 114 & 0.99 & 2.70203 & 0.99000 & 0.00000 & 0.00014 & 0.014 & 0.07\tabularnewline
\hline 
\end{tabular}
\par\end{centering}
\caption{\protect\label{tab:numerics}QF of the distributions TS, N(0,1) and
NIG. The parameters $a$, $b$ and $N$ of the COS method for NIG
and N(0,1) are obtained as described in Section \ref{sec:Theory}.
In the case of the TS distribution, we set $N=50$.}
\end{table}

We provide an example how Theorem \ref{thm:HH-1} can be used to choose
the error tolerance $\varepsilon$ for the NIG CDF to approximate
the NIG QF arbitrarily closely. (Without calculating any reference
values). Suppose $p=0.99$ and $\delta=0.1$ is the error tolerance
for the QF. First, set $\varepsilon_{1}:=0.005$ and obtain $a(\varepsilon_{1})=-5.9$,
$b(\varepsilon_{1})=5.9$ and $N(\varepsilon_{1})=79$ as described
in Section \ref{sec:Theory}. From these parameters, compute $H_{1}$
and $h_{1}$. By numerical inversion using the bisection method, we
get $y_{1}:=H_{\text{Num},1}^{-1}(p)=2.7$. We see that the RHS of
Inequality (\ref{eq:H-1}) is (approximately) equal to $\frac{2\varepsilon_{1}}{\min\{h_{1}(y_{1}\pm\varepsilon_{1})\}}+\varepsilon_{1}=0.73>\delta$.
So, $\varepsilon_{1}$ is too large. In the next step, we set $\varepsilon_{2}:=0.0005\approx\frac{\delta}{\frac{2}{\min\{h_{1}(y_{1}\pm\varepsilon_{1})\}}+1}$
and obtain $a(\varepsilon_{2})=-7.9$, $b(\varepsilon_{2})=7.9$ and
$N(\varepsilon_{2})=114$. We observe in Table \ref{tab:numerics}
that the RHS of Inequality (\ref{eq:H-1}) is now satisfied, and Theorem
\ref{thm:HH-1} ensures that $|H_{\text{Num},2}^{-1}(p)-F^{-1}(p)|<\delta$. 

Finally, a word about computational time. Note that the formulas for
$a$, $b$ and $N$ in Section \ref{sec:Theory} do not depend on
$y$ or $p$ and need to be computed only once. The COS method must
evaluate the CF $N$ times, which is extremely fast since $\varphi$
is given in closed form. For example, using the R software and vectorized
code on an Intel i7-10750H CPU computing the CDF of the TS distribution
with $N=50$ takes on average 11 microseconds.

\bibliographystyle{plainnat}
\bibliography{biblio}

\end{document}